\documentclass[12pt]{article}

\usepackage{amsmath}
\usepackage{amsthm}
\usepackage{amssymb}
\usepackage{graphicx}
\usepackage{color}
\usepackage{bm}
\usepackage{url}
\usepackage{array}
\usepackage{multirow}
\usepackage{float}

\allowdisplaybreaks[4]

\title{\bf
A sufficient condition for existence of iterative roots of PM functions
without the characteristic endpoints
condition\thanks{The author Xiao Tang was supported by NSFC grant \#12001537 and 20XLB033. The author Lin Li was supported by NSFC grant $\#$11671061 and  Zhejiang Provincial Natural Science
	Foundation of China under Grant No.LY18A010017.}
}

\author{
{\normalsize
{\sc Xiao Tang}\,$^{a}$,
 ~~~{\sc Lin Li}\,$^{b}$\thanks{Corresponding author: mathll@163.com}
       }
\\
$^{a}${\small  School of Mathematical Science, Chongqing Normal University, }
\\
{\small Chongqing 401331, P.R. China}

\\
$^{b}${\small  College of Mathematics, Physics and Information Engineering, }
\\
{\small Jiaxing University, Jiaxing 314001, P.R. China}

}

\date{}

\begin{document}
	
\maketitle

	\numberwithin{equation}{section}
	\newtheorem{lm}{Lemma}[section]
	\newtheorem{thm}{Theorem}[section]
	\newtheorem{remk}{Remark}
	\newtheorem{exam}{Example}
	\newtheorem{prop}{Proposition}
	\newtheorem{cor}{Corollary}
	\newtheorem{defn}{Definition}
	
\begin{abstract}
For PM functions of height 1, the existence of continuous iterative roots of any order was obtained under the characteristic endpoints condition. This raises an open problem about iterative roots without this condition, called characteristic endpoints problem. This problem is solved almost completely when the number of forts is equal to or less than the order. In this paper, we study the case that the number of forts is greater than the order and give a sufficient condition for existence of continuous iterative roots of order $2$ with height 2, answering the characteristic endpoints problem partially.
\end{abstract}

		

	\setlength\arraycolsep{4pt}
	\baselineskip 16pt
	\parskip 0.4cm


\section{Introduction}
Let $I:=[a,b]\subset\mathbb{R}$ be a compact interval. Given a continuous function $F:I\to I$, if there exists a continuous function $f:I\to I$ such that
\begin{equation}
f^n(x)=F(x),~~~~~\forall x\in I,
\end{equation}
where $n>1$ is an integer, $f^n$ is $n$-th iterate of $f$ and defined recursively by $f^n(x):=f(f^{n-1}(x))$ and $f^0(x):=x$ for all $x\in I$,  then we call that $f$ is a continuous {\it iterative root} of $F$ of order $n$.

In the theory of dynamical systems, embedding flow is an important problem which is a bridge between discrete and continuous dynamical systems (\cite{Baron-Jar,Fort,Irwin,PM,TS,Zdun-Solarz}). As a weak vision of embedding flow, iterative roots are studied extensively, especially in 1-dimensional case (see \cite{Babbage,Kuczma61,Solarz,ZY,LYZ,LJLZ} and references therein).

It is known that there are plentiful results on iterative roots of continuous monotone self-mappings (\cite{Kuczma,KCG}). As a generalization of monotone functions, a kind of non-monotone functions with finite non-monotone points was studied since 80's (\cite{ZY,Zhang97}). A continuous self-mapping $F:I\to I$ is said to be a {\it piecewise monotone function} (abbreviated as {\it PM function}) if $F$ has finitely many non-monotone points (or forts), named by $\{c_i\}_{i=1}^{v}$ for some $v\in \mathbb{N}$ such that
$$
c_0:=a<c_1<c_2<\cdots<c_v<c_{v+1}:=b.
$$
Clearly, $F$ is strictly monotone on each subinterval $[c_i,c_{i+1}]$ for  $0\le i\le v$ and such subinterval $[c_i,c_{i+1}]$ is called a {\it lap} of $F$. Let $N(F)$ be the number of forts of $F$. It is proved (\cite{ZY,Zhang97}) that the sequence $\{N(F^i)\}_{i\in \mathbb{N}}$ is nondecreasing, that is,
$$
0=N(F^0)\le N(F)\le N(F^2)\le \cdots \le N(F^i)\le \cdots.
$$
Furthermore, if $N(F^{i_0})=N(F^{i_0+1})$ for some integer $i_0\ge0$, then $N(F^{i_0})=N(F^{i_0+j})$ for all integers $j\ge1$. The smallest integer $i_0\in \mathbb{N}$ such that $N(F^{i_0})=N(F^{i_0+1})$ is called {\it nonmonotonicity height} $H(F)$ (simply {\it height}) of $F$ if such a $i_0$ exists and $H(F)=\infty$ otherwise. We denote the set of all piecewise monotone functions  by ${\cal PM}(I,I)$.
One can prove that there exists a lap, called {\it characteristic interval} $K(F)$ of $F$, that covers the range of $F$ if and only if $H(F) =1$. An important result about continuous iterative roots of PM functions with height 1 is the following.
\begin{center}
\begin{minipage}{128mm}
\noindent
{\bf Theorem A} (Theorem 4 in \cite{Zhang97}).
{\it
Let $F\in {\cal PM}(I,I)$ be of
height 1. Suppose that
\begin{description}
\item [(K$^+$) ] $F$ is strictly increasing
on its characteristic interval $K(F)=[a',b']$, and that

\item[(K$^+_0$)] $F$ on $I$ cannot reach $a'$ and $b'$ unless $F(a')=a'$ or
$F(b')=b'$.
\end{description}
Then,
for any integer $n>1$, $F$ has continuous iterative roots of order $n$.
Conversely,
conditions {\bf  (K$^+$)} and {\bf  (K$^+_0$)} are necessary for $n>N(F)+1$.
}
\end{minipage}
\end{center}
The converse part of this theorem suggests an open problem
(see \cite{ZY} or \cite{Zhang97}):
{\it Does a function $F\in {\cal PM}(I,I)$ with $H(F)= 1$ have an iterative root of an order $n\leq N(F)+1$
if condition {\bf  (K$^+_0$)}, called  `characteristic endpoints condition', is not satisfied?
}
The open problem is answered by the second author and Zhang (\cite{LZ}) in many cases. They discussed existence and nonexistence results of continuous iterative roots, which are shown in Table \ref{TnoCECI}. 
\newcommand{\tabincell}[2]{\begin{tabular}{@{}#1@{}}#2\end{tabular}}
\begin{table}
	\centering
	\renewcommand\arraystretch{1.5}
	\begin{tabular}{|c|c|c|}
		\hline
		Order &  Properties of $f$ & \tabincell{c}{Existence or\\ Nonexistence} \\
		\hline
	\multirow{2}*{$n= 2$}
		 & $f$  on $[a',b'] \uparrow$, $H(f)=1$ & Nonexistence \\ \cline{2-3}
		&  $f$  on $[a',b'] \downarrow$, $H(f)=1$ & Existence
		\\
		\hline
	\multirow{2}*{$2<n<N(F)$}
		 & $f$  on $[a',b'] \uparrow$, $H(f)<n$ & Nonexistence
		\\
		\cline{2-3}
		 & $f$  on $[a',b'] \downarrow$, $H(f)<n-1$ & Nonexistence 
		\\
		\hline
		\multirow{3}*{$n=N(F)$}
		 & $f$  on $[a',b'] \uparrow$, $H(f)\le n$ &  Nonexistence 
		\\
		\cline{2-3}
		 & $f$  on $[a',b'] \downarrow$, $H(f)<n-1$ & Nonexistence
		\\
		\cline{2-3}
	 & $f$  on $[a',b'] \downarrow$, $H(f)=n$ & Existence 
		\\
		\hline
	\multirow{2}*{$n=N(F)+1$}
	 & $f$  on $[a',b'] \uparrow$, $H(f)\le n$ &Nonexistence
		\\
	\cline{2-3}
 & $f$  on $[a',b'] \downarrow$, $H(f)<n-1$ & Nonexistence 
		\\
		\hline
	\end{tabular}
	\caption{The results of iterative roots for $F$ being strictly increasing on $[a',b']$  without the characteristic endpoints condition, where the symbols $\uparrow$ and $\downarrow$ denote strictly increasing and strictly decreasing, respectively.}
	\label{TnoCECI}
\end{table}

In this paper, because of complexity of iteration, we continue to study the open problem only in the case of order $n=2$. According to \cite[Lemma 2]{LYZ}, we only need to consider the case $H(f)=n=2$. Moreover, the number $N(F)$ of forts of $F$ must be greater than $2$. In fact, when $N(F)=1$, if $F$ has a continuous iterative roots of order $2$, then we get $N(F)=N(f^2)>N(f)\ge 1$, implying that $N(F)\ge 2$, a contradiction to the fact that $N(F)=1$. As shown in Table \ref{TnoCECI}, case $N(F)=n=2$ is solved completely. Thus, the case we studied is $H(f)=n=2<N(F)$,
 which is not  considered in \cite{LZ}.

This paper is organised as follows. In section 2, we
give  a sufficient condition for existence of continuous iterative roots which are  strictly increasing on $K(F)$
and also give a sufficient condition for  existence of continuous ones 
which are  strictly decreasing on $K(F)$.
Section 3 is devoted to present
some auxiliary lemmas. Finally, we give the proofs of our Theorems in section 4.

\section{Main results}

When $F$ is strictly increasing on its characteristic interval $[a',b']$, it has two classes of continuous iterative roots: strictly increasing on $[a',b']$ and strictly decreasing on $[a',b']$. We will discuss these two classes of  iterative roots of order $2$ with height 2 in the following two subsections, separately. We only give our results in this section and put their proofs in section 4.

	
	

\subsection{Roots increasing on characteristic interval}
The following Theorem \ref{LI} deals with case the characteristic interval of $F$ is  the first lap, that is, $K(F)=[c_0,c_1]$.

\begin{thm}
Let $F\in {\cal PM}(I,I)$ with $H(F)=1$ and $K(F)=[c_0,c_1]$ such that {\bf  (K$^+$)} holds. Suppose that $F(c_3)= \max_{x\in[c_0,c_u]}F(x)$, where $c_u\in \{c_4,c_5,...,c_{v+1}\}$ is the smallest point such that $F(c_u)=c_0$, and  that $F$ satisfies either
\begin{align}
F(c_2)=c_0=F(c_0),~~~~~ &F(c_1)<F(c_3)<c_1=M~~\text{or},
\label{lf01}\\
c_0<F(c_2)<F(c_0),~~~~~ &F(c_3)=c_1=F(c_1)~~\text{or},
\label{lf02}
\\
c_0<F(c_2)<F(c_0),~~~~~ &F(c_1)<F(c_3)<c_1,
\label{lf04}
\end{align}
where $M=\max_{x\in I} F(x)$. Then $F$ has a continuous iterative root $f$ of order $2$ with
$H(f)=2$ such that $f$ is strictly increasing on $[c_0,c_1]$.
\label{LI}
\end{thm}

Clearly, $F$ satisfying one of conditions (\ref{lf01})-(\ref{lf04}) with condition $F(c_u)=c_0$ for $c_u\in \{c_4,c_5,...,c_{v+1}\}$ in Theorem~\ref{LI} does not satisfy {\bf  (K$^+_0$)}. In order to illustrate these conditions in Theorem~\ref{LI}, we give an example in Figure 1. From Figure 1, it is easy to see that $F\in {\cal PM}(I,I)$, $H(F)=1$, $K(F)=[c_0,c_1]$, and $F$ is strictly increasing on $[c_0,c_1]$. Furthermore, $F$ satisfies condition \eqref{lf04} and $c_6$ is the smallest point in $\{c_4,c_5,...,c_8\}$ such that $F(c_6)=c_0$ and $F(c_3)=\max F|_{[c_0,c_6]}$. Therefore, $F$ shown in Figure \ref{ilf} satisfies the conditions of Theorem \ref{LI}.

\begin{figure}[H]
	\begin{minipage}[t]{0.5\textwidth}
		\includegraphics[scale=0.3]{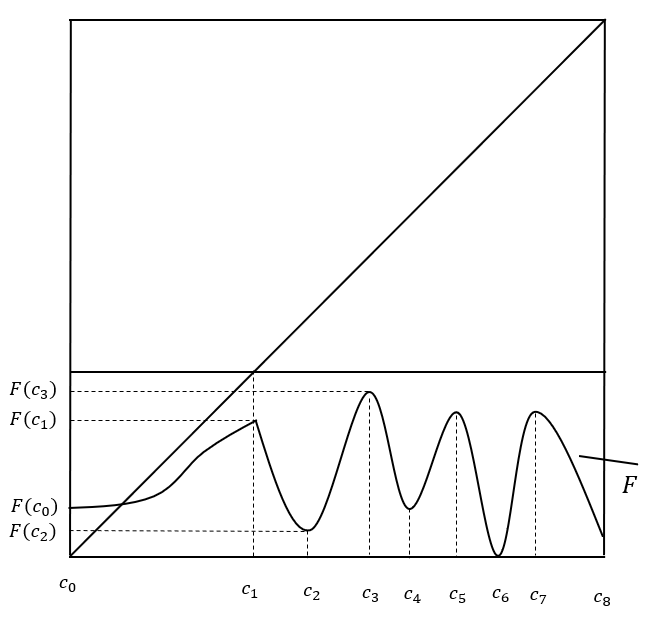}
		\caption{$F$ satisfies conditions of \protect\\Theorem \ref{LI}.}
		\label{ilf}
	\end{minipage}
	\begin{minipage}[t]{0.5\textwidth}
		\includegraphics[scale=0.3]{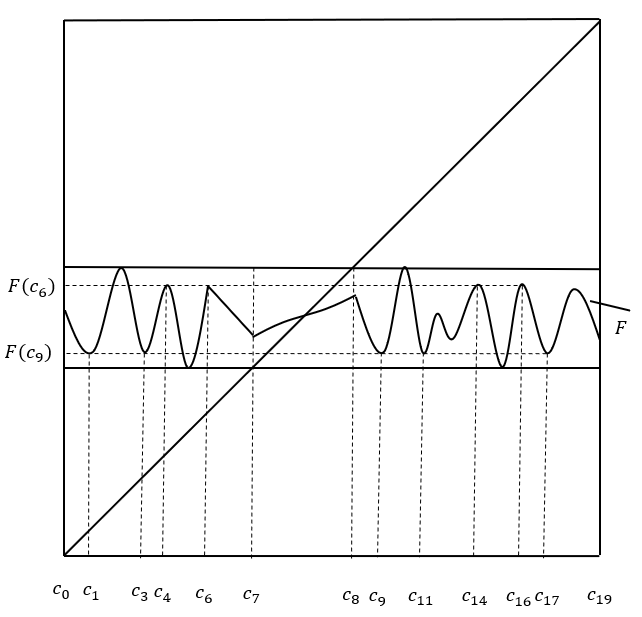}
		\caption{$F$ satisfies conditions of \protect\\ both Theorems \ref{MI} and \ref{MD}.}
		\label{imf}
	\end{minipage}
\end{figure}

When the characteristic interval of $F$ is the last lap, that is, $K(F)=[c_v,c_{v+1}]$, by a change of coordinates $\varphi(x)=-x+a+b$, we get a transformed map $G:=\varphi^{-1}\circ F\circ \varphi.$ Clearly, $G\in{\cal PM}(I,I)$, $H(G)=1$ and the  characteristic interval of $G$ is the first lap of $G$. Thus, according to Theorem \ref{LI}, by imposing corresponding conditions on $G$, the function $G$ has a continuous iterative root $g$ of order $2$ with $H(f)=2$ such that $g$ is strictly increasing on the characteristic interval of $G$. Then
$H(\varphi\circ g\circ \varphi^{-1})=2$, $\varphi\circ g\circ \varphi^{-1}$ is strictly increasing on $[c_v,c_{v+1}]$ and is a continuous iterative root of order $2$ of $F$.

Next, we consider the case that the characteristic interval of $F$ neither is the first lap nor the last one, that is, $K(F)=[c_k,c_{k+1}]$ for some $1\le k\le v-1$.

\begin{thm}
Let $F\in{\cal PM}(I,I)$ with $H(F)=1$ and $K(F)=[c_k,c_{k+1}]$ for some $1\le k\le v-1$ such that {\bf  (K$^+$)} holds but  {\bf  (K$^+_0$)} does not hold. Suppose that $F$ satisfies either
\begin{align}
&c_k<F(c_{k+2})<F(c_k),~~~F(c_{k-1})=c_{k+1}=F(c_{k+1}),~~\text{or}
			\label{f_01}
			\\
			&c_k=F(c_k)=F(c_{k+2}), ~~~F(c_{k+1})<F(c_{k-1})<c_{k+1},~~\text{or}
			\label{f_02}
			\\
			&c_k<F(c_{k+2})<F(c_k), ~~~F(c_{k+1})<F(c_{k-1})<c_{k+1}.
			\label{f_03}
			\end{align}
Suppose further that there exist forts
			\begin{equation}
\label{cc}
			c_{\ell_0}:=c_{k-1}>c_{\ell_1}>c_{\ell_2}>\cdots>c_{\ell_s}> c_{\ell_{s+1}}:=c_0
			\end{equation}
			such that
\begin{equation}
			F(c_{\ell_j})=\begin{cases}
			F(c_{k-1})~~~&\text{for}~~~ j\equiv 0 ~\text{or}~ 1(\!\!\!\!\mod 4),\\
			F(c_{k+2})~~~&\text{for}~~~ j\equiv 2 ~\text{or}~ 3(\!\!\!\!\mod 4),
			\end{cases}
			\label{jointp}
			\end{equation}
except for $j=s+1$,
			\begin{align}
			&\max_{x\in [c_{\ell_{j+1}},c_{\ell_j}]} F(x) =F(c_{k-1})~~\text{for}~~ j\equiv 0,1 ~\text{or}~ 3(\!\!\!\!\mod 4),
			\label{upexten1}
			\\&
	\max_{x\in [c_{\ell_{j+1}},c_{\ell_j}]} F(x)=F(c_{k+2})~~\text{for}~~ j\equiv 1,2 ~\text{or}~ 3(\!\!\!\!\mod 4),
			\label{midexten}
			\end{align}
and forts
			\begin{equation}
			c_{r_0}:=c_{k+2}<c_{r_1}<c_{r_2}<\cdots<c_{r_t}<c_{r_{t+1}}:=c_{n+1}
			\end{equation}
			such that
			\begin{equation}
			F(c_{r_j})=\begin{cases}
			F(c_{k+2})~~~&\text{for}~~~ j\equiv 0 ~\text{or}~ 1(\!\!\!\!\mod 4),\\
			F(c_{k-1})~~~&\text{for}~~~ j\equiv 2 ~\text{or}~ 3(\!\!\!\!\mod 4),
			\end{cases}
			\label{extenck+2cr1}
			\end{equation}
			except for $j=t+1$,
\begin{align}
			&\min_{x\in [c_{r_j},c_{r_{j+1}}]} F(x)=F(c_{k+2})~~\text{for}~~ j\equiv0, 1 ~\text{or}~ 3(\!\!\!\!\mod 4),
			\label{midextenr1}
			\\&
\min_{x\in [c_{r_j},c_{r_{j+1}}]} F(x)=F(c_{k-1})~~\text{for}~~ j\equiv1, 2~\text{or}~ 3(\!\!\!\!\mod 4).
			\label{midextenr2}
			\end{align}
			Then $F$ has a continuous  iterative root $f$ of order $2$ with $H(f)=2$ such that $f$ is strictly increasing on $[c_k,c_{k+1}]$.
			\label{MI}
		\end{thm}

In order to illustrate these conditions in Theorem~\ref{MI}, we also give an example in Figure 2. From Figure 2, one can check that $F\in{\cal PM}(I,I), H(F)=1, K(F)=[c_7,c_8]$ and $F$ is strictly increasing on $[c_7,c_8]$. Since $F(c_5)=F(c_{15})=c_7<F(c_7)$ and $F(c_2)=F(c_{10})=c_8>F(c_8)$, we conclude that $F$ does not satisfy the characteristic endpoints condition. Moreover, it is easy to see that $F$ satisfies  \eqref{f_03} and there exist forts
\begin{equation*}
c_{\ell_0}:=c_6>c_{\ell_1}:=c_4>c_{\ell_2}:=c_3>c_{\ell_3}:=c_1> c_{\ell_4}:=c_0
\end{equation*}
such that $F(c_4)=F(c_6), F(c_1)=F(c_3)=F(c_9)$,
\begin{align*}
&\max_{x\in [c_0,c_1]} F(x)=\max_{x\in [c_3,c_4]} F(x)=\max_{x\in [c_4,c_6]} F(x)=F(c_6),\\
&\min_{x\in [c_0,c_1]} F(x)=\min_{x\in [c_1,c_3]} F(x)=\min_{x\in [c_3,c_4]} F(x)=F(c_9),
\end{align*}
and forts
\begin{equation*}
c_{r_0}:=c_9<c_{r_1}:=c_{11}<c_{r_2}:=c_{14}<c_{r_3}:=c_{16}<c_{r_4}:=c_{17}<c_{r_5}:=c_{19}
\end{equation*}
such that $F(c_9)=F(c_{11})=F(c_{17}), F(c_6)=F(c_{14})=F(c_{16})$,
\begin{align*}
&\min_{x\in[c_9,c_{11}]} F(x)=\min_{x\in[c_{11},c_{14}] }F(x)=\min_{x\in[c_{16},c_{17}]} F(x)=\min_{x\in[c_{17},c_{19}]} F(x)=F(c_9),\\
&\max_{x\in[c_{11},c_{14}]} F(x)=\max_{x\in[c_{14},c_{16}]} F(x)=\max_{x\in[c_{16},c_{17}]} F(x)=F(c_6).
\end{align*}
The discussion above shows that $F$ whose graph given in Figure \ref{imf} satisfies the conditions of Theorem \ref{MI}.

\subsection{Roots decreasing on characteristic interval}

Before presenting main results in this subsection, we introduce the concept of reversing-correspondence, which can be found in \cite{Kuczma,LJLZ}.
\begin{defn}
Let $[\alpha,\beta]\subset\mathbb{R}$ be a compact interval. A strictly increasing continuous function $\phi:[\alpha,\beta]\to[\alpha,\beta]$  is said to be reversing-correspondence if
\begin{description}
\item [(i)] there exists a point $\xi\in Fix\phi$ and a strictly decreasing map $\omega$ mapping $Fix\phi$ onto itself such that $\omega(\xi)=\xi$, where $Fix\phi$ denotes the set of all fixed points of $\phi$,
\item[(ii)] $\alpha$ and $\beta$ both belong to $Fix\phi$ or not,
\item[(iii)] for every pair of consecutive fixed points $\xi_1$ and $\xi_2$ such that $\xi_1<\xi_2\le \xi$, we have $(\phi(x)-x)(\phi(y)-y)<0$ for $(x,y)\in(\xi_1,\xi_2)\times(\omega(\xi_2),\omega(\xi_1))$.
\end{description}
\end{defn}

Same as discussion for roots increasing on characteristic interval, we first consider the case that the first lap is the characteristic interval, that is, $K(F)=[c_0,c_1]$.

\begin{thm}
Let $F\in {\cal PM}(I,I)$ with $H(F)=1$ and $K(F)=[c_0,c_1]$ such that {\bf  (K$^+$)} holds.  Suppose that $F$ is  reversing-correspondence on $[c_0,c_1]$ and that
\begin{equation}
F(c_0)>c_0, ~~~F(c_1)<c_1~~~ and~~~ F(c_2)=c_0.
\label{ldc}
\end{equation}
Then $F$ has a continuous iterative root $f$ of order $2$ with $H(f)=2$ such that $f$ is strictly decreasing on $[c_0,c_1]$.
	\label{LD}
\end{thm}

According to condition \eqref{ldc}, it is easy to see that $F$ in Theorem \ref{LD} does not satisfy  {\bf  (K$^+_0$)}. We also give an example in Figure 3, which satisfies all conditions of Theorem \ref{LD}.

Similarly, when the last lap $[c_v,c_{v+1}]$ is the characteristic interval of $F$, by a change of coordinates $\varphi(x)=-x+a+b$, Theorem \ref{LD} can be applied to the map $G:=\varphi^{-1}\circ F\circ \varphi$, whose characteristic interval is the first lap.  Then, we can employ iterative roots of $G$ to give iterative roots of $F$.

\begin{figure}[H]	
	\centering
	\includegraphics[scale=0.53]{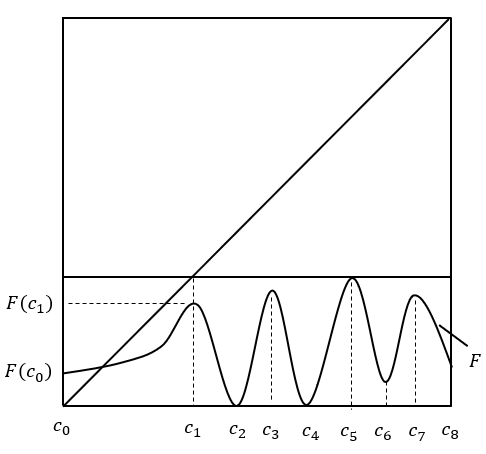}
	\caption{$F$ satisfies conditions of  Theorem \ref{LD}. }
	\label{dlf}
\end{figure}

Next, we consider the case that $K(F)=[c_k,c_{k+1}]$ for some $1\le k\le v-1$.

\begin{thm}
Let $F\in{\cal PM}(I,I)$ with $H(F)=1$ and $K(F)=[c_k,c_{k+1}]$ for some $1\le k\le v-1$ such that {\bf  (K$^+$)} holds but  {\bf  (K$^+_0$)} does not hold. Suppose that $F$ is  reversing-correspondence on $[c_k,c_{k+1}]$.
Suppose further that $F$ satisfies either
\begin{align}
c_k=F(c_{k+2})&<F(c_k)<F(c_{k+1})=F(c_{k-1})<c_{k+1},~~~\text{or}
\label{mdc}
\\
c_k<F(c_{k+2})&<F(c_k)<F(c_{k+1})<F(c_{k-1})<c_{k+1},~~~\text{or}
\label{mdc2}
\\
c_k<F(c_{k+2})&=F(c_k)<F(c_{k+1})<F(c_{k-1})=c_{k+1}.
\label{mdc3}
\end{align}
Then $F$ has a continuous iterative root $f$ of order $2$ with $H(f)=2$  such that $f$ is strictly decreasing on $[c_k,c_{k+1}]$ if the conditions \eqref{cc}-\eqref{midextenr2} in Theorem $\ref{MI}$ hold.
\label{MD}
\end{thm}

One can check that the PM function $F$ given in Figure \ref{imf} also satisfies conditions of Theorem \ref{MD}. In fact, from Figure 2, we see that $F(c_7)>c_7$, $F(c_8)<c_8$ and $F$ has only one fixed point on $[c_7,c_8]$, implying that $F|_{[c_7,c_8]}$ is reversing-correspondence on $[c_7,c_8]$ by Definition $1$. Furthermore, note that $F(c_9)<F(c_7)<F(c_8)<F(c_6)$, that is, $F$ satisfies \eqref{mdc2}. Finally, as shown in the paragraph just after Theorem \ref{MI}, $F$ satisfies the conditions \eqref{cc}-\eqref{midextenr2}.

\section{Auxiliary lemmas}

In this section, we give some lemmas that are helpful to our proofs in section 4.

\begin{lm}
	Let $\varPhi:[\alpha,\beta]\to[\alpha,\beta]$, where $\alpha,\beta\in \mathbb{R}$, be a continuous and strictly increasing function.
	If $\varPhi(x)> x ~(resp.~ \varPhi(x)< x)$ for all $x\in [\alpha,\beta) ~(resp. ~x\in (\alpha,\beta])$, then, for any $x_*\in(\alpha,\varPhi(\alpha))~(resp.~ y_*\in(\varPhi(\beta),\beta))$, the function	$\varPhi$ has a continuous and strictly increasing iterative root $\varphi$ of order $2$ such that $\varphi(\alpha)=x_*~(resp.~ \varphi(\beta)=y_*)$.
	\label{iev}
\end{lm}

\begin{proof}[\bf Proof]
We only prove the case that $\varPhi(x)> x $ for all $x\in [\alpha,\beta)$ because the proof for the other case $\varPhi(x)< x $ is similar.	
	Let $x_0:=\alpha$, $x_1:=x_*$ and $\varPhi^i(x_0)=x_{2i}$, $\varPhi^i(x_1)=x_{2i+1}$ for $i\ge 1$. Since $\varPhi$ is strictly increasing and $\varPhi(x)>x$	for all $x\in[\alpha,\beta)$, the sequence $\{x_i\}_{i\ge0}$ is also strictly increasing. Note that $\varPhi$ is continuous, it follows that $\varPhi(\beta)=\beta$, implying $x_i\to \beta$ as $i\to +\infty$. Fix a strictly increasing and continuous bijection $\varphi_0:[x_0,x_1] \to [x_1,x_2]$. Then $\varphi_0$ can be extended to a strictly increasing and continuous  iterative root $\varphi$ of order $2$ of $\varPhi$ on $[\alpha,\beta]$ uniquely. In fact, the iterative root $\varphi:[\alpha,\beta]\to [\alpha,\beta]$ of order $2$ can be defined as
	\begin{equation}
	\varphi(x):=\begin{cases}
	\varphi_i(x),~~~&x\in [x_i,x_{i+1}],~i\ge0,\\
	\beta,~~~&x=\beta,
	\end{cases}
	\end{equation}
	where $\varphi_i: [x_i,x_{i+1}]\to [x_{i+1},x_{i+2}]$ is defined recursively by
	$$
	\varphi_i(x):=\varPhi\circ \varphi_{i-1}^{-1}(x) ~~~\text{for}~~~x\in[x_i,x_{i+1}],~i\ge1.
	$$
	The proof is completed.
\end{proof}

\begin{lm}
	Let $\varPhi:[\alpha,\beta]\cup[\eta,\xi] \to [\alpha,\beta]\cup[\eta,\xi]$ be a continuous and strictly increasing function, where $\alpha, \beta, \eta, \xi\in \mathbb{R}$ and $\beta\le\eta$. Suppose that $\varPhi(\alpha)>\alpha,\varPhi(\beta)=\beta,\varPhi(\eta)=\eta$ and $\varPhi(\xi)<\xi$. Then, for any $(x_*,y_*)\in \{ (\varPhi(\alpha),\xi)\}\cup (\alpha,\varPhi(\alpha)) \times (\varPhi(\xi),\xi) \cup \{(\alpha,\varPhi(\xi)) \}$, the function $\varPhi$ has a continuous and strictly decreasing iterative root $\varphi$ of order $2$ such that $\varphi(\alpha)=y_*$, $\varphi(\beta)=\eta$ and $\varphi(\xi) =x_*$.
	\label{diteendps}
\end{lm}

\begin{proof}[\bf Proof.]
	We only give a proof for $(x_*,y_*)\in(\alpha,\varPhi(\alpha)) \times (\varPhi(\xi),\xi)$ because the discussion for the remaining two cases is similar.
	Fix a continuous and strictly decreasing function $\varphi_0:[\alpha,\varPhi(\alpha)]\to[\varPhi(y_*),y_*]$ such that
	\begin{equation}
	\varphi_0(\alpha)=y_*,~~~\varphi_0(\varPhi(\alpha))=\varPhi(y_*) ,~~~\varphi_0(x_*)=\varPhi(\xi).
	\label{endcondi}
	\end{equation}
	Clearly, $\varphi_0$ is well defined because $x_*\in (\alpha,\varPhi(\alpha))$ and $\varPhi(\xi)\in (\varPhi(y_*),y_*)$.
	With the aid of [\cite{KCG}, Theorem 5.3.1], the function $\varphi_0$ can be extended to a continuous and strictly decreasing solution $\widehat{\varphi}:[\alpha,\beta]\to[\eta,y_*]$ of the equation
	\begin{equation}
	\widehat{\varphi}(\varPhi(x))=\varPhi(\widehat{\varphi}(x)),~~~x\in [\alpha,\beta],
	\label{exchaequa}
	\end{equation}
	satisfying $\widehat{\varphi}(\beta)=\eta$. In order to make the proof more readable, we give a detail of the extension as follows. For every $x\in[\varPhi^i(\alpha),\varPhi^{i+1}(\alpha)]$, $i\ge1$, define
	\begin{equation}
	\varphi_i(x):=\varPhi^i\circ \varphi_0\circ \varPhi^{-i}(x).
	\end{equation}
	Clearly, $\varphi_i$'s are continuous and strictly decreasing for all $i\ge1$. By \eqref{endcondi}, one can check that
	$\varphi_i(\varPhi^i(\alpha))=\varPhi^i(y_*)$ and $\varphi_i(\varPhi^{i+1}(\alpha))=\varPhi^{i+1}(y_*)$ for all $i\ge1$. Thus, the function
	\begin{equation}
	\widehat{\varphi}(x):=
	\begin{cases}
	\varphi_i(x),~~~& x\in [\varPhi^i(\alpha),\varPhi^{i+1}(\alpha)],~i\ge1,\\
	\eta,~~~&x=\beta,
	\end{cases}
	\label{hatf}
	\end{equation}
	is well defined and continuous on $[\alpha,\beta]$. Note that for every $x\in[\alpha,\beta)$, there exists an integer $i$ such that $x\in[\varPhi^i(\alpha),\varPhi^{i+1}(\alpha)]$ and then
	\begin{equation*}
	\widehat{\varphi}(\varPhi(x)) = \varphi_{i+1}(\varPhi(x))=\varPhi^{i+1}\circ \varphi_0\circ \varPhi^{-i-1}(\varPhi(x))=\varPhi(\varPhi^i\circ \varphi_0\circ \varPhi^{-i}(x))=\varPhi(\widehat{\varphi}(x)).
	\end{equation*}
	Moreover, $\widehat{\varphi}(\varPhi(\beta))=\widehat{\varphi}(\beta)=\eta=\varPhi(\widehat{\varphi}(\beta))$. This proves that $\widehat{\varphi}$ defined by \eqref{hatf} is a continuous and strictly decreasing solution of equation \eqref{exchaequa}.
	
	Now, we define a function as desired in this Lemma. Let
	\begin{equation}
	\varphi(x):=\begin{cases}
	\widehat{\varphi}(x),~~~&x\in [\alpha,\beta],\\
	\widehat{\varphi}^{-1}\circ \varPhi(x),~~~&x\in [\eta,\xi].
	\end{cases}
	\label{df}
	\end{equation}
	Clearly, $\widehat{\varphi}^{-1}\circ \varPhi$ is continuous and strictly decreasing from $[\eta,\xi]$ to $[x_*,\beta]$. For every $x\in[\alpha,\beta]$, by \eqref{exchaequa} and \eqref{df} we have $\varphi^2(x)=\widehat{\varphi}^{-1}\circ \varPhi\circ \widehat{\varphi}(x) =\widehat{\varphi}^{-1}\circ\widehat{\varphi}\circ \varPhi(x)=\varPhi(x)$. For every $x\in[\eta,\xi]$, by \eqref{df} we have $\varphi^2(x)=\widehat{\varphi}\circ\widehat{\varphi}^{-1}\circ \varPhi(x)=\varPhi(x)$. It is shown that the function $\varphi$ defined in \eqref{df} is a continuous and strictly decreasing iterative root of $\varPhi$ of order $2$. From \eqref{df}, \eqref{hatf} and the first equality of \eqref{endcondi}, we have $\varphi(\alpha)=y_*$ and $\varphi(\beta)=\eta$. Moreover, by the third equality of \eqref{endcondi}, we obtain
	$$
	\varphi(\xi)=\widehat{\varphi}^{-1}\circ \varPhi(\xi)=\varphi_0^{-1}\circ \varPhi(\xi)=x_*.
	$$
	The proof is completed.
\end{proof}

\begin{lm}[Lemma 3.1 in \cite{LC}]
	Let $F\in{\cal PM}(I,I)$ and $H(F)=m\in (0,+\infty)$. Then $m$ is the smallest nonnegative integer such that $F^m(I) \subset K(F^m)$, where $K(F^m)$ denotes the characteristic interval of $F^m$.
	\label{height}
\end{lm}

\section{Proofs of Theorems}
In this section, we prove Theorems~\ref{LI}-\ref{MD}, as given in section $2$.
\begin{proof}[\bf Proof of Theorem \ref{LI}]
{\it Step 1}: Construct a continuous and strictly increasing  iterative root of order $2$ of $F$ on $[c_0,c_1]$.

Under condition  \eqref{lf01}, \eqref{lf02} or \eqref{lf04}, by \cite[Theorem $15.7$]{Kuczma} and Lemma \ref{iev}, we can construct
	a continuous and strictly increasing function $f_0:[c_0,c_1]\to[c_0,c_1]$ satisfying
	\begin{equation}
	f_0(c_0)=F(c_2),~~~~~ f_0(c_1)=F(c_3)
	\label{extencon}
	\end{equation}
	such that $f_0^2(x)=F(x)$ for all $x\in [c_0,c_1]$.

{\it Step 2}: Define functions $f_i$ on each $[c_i,c_{i+1}]$ for $1\le i\le v$.

For $i=1,2,$ define
$f_i:[c_i,c_{i+1}]\to [c_0,c_1]$ by
	\begin{equation}
	f_i(x):=f_0^{-1}\circ F(x),~~~~~x\in[c_i,c_{i+1}].
	\label{f12}
	\end{equation}
	By \eqref{extencon}, one can check that $F([c_i,c_{i+1}])(i=1,2)$ are both included in the range of $f_0$ because $F(c_3)>F(c_1)$, implying that $f_i (i=1,2)$ are well defined.
	Clearly, $f_1$ and $f_2$ are both continuous and strictly monotone on their domains. Moreover,
	by \eqref{extencon} we have
	\begin{align}
	&f_1(c_1)=f_0^{-1}\circ F(c_1)=f_0^{-1}\circ f_0^2(c_1)=f_0(c_1),
	\label{f1c1}
	\\
	&f_1(c_2) =f_2(c_2)= f_0^{-1}\circ F(c_2)= c_0,
	\label{f1c2}
	\\
	&f_2(c_3)=f_0^{-1}\circ F(c_3)=c_1.
	\label{f2c3}
	\end{align}
	For $3\le i\le u-1$, define $f_i:[c_i,c_{i+1}]\to [c_1,c_2]$ by
	\begin{equation}
	f_i(x):=f_1^{-1}\circ F(x),~~~~~x\in[c_i,c_{i+1}].
	\label{f1f}
	\end{equation}
	Since $F(c_3)=\max_{x\in[c_0,c_u]} F(x)$, we have $$F([c_i,c_{i+1}])\subseteq [c_0,F(c_3)]=[f_1(c_2),f_1(c_1)]
	$$ for each $i=3,\cdots,u-1$ by \eqref{extencon}, \eqref{f1c1} and \eqref{f1c2}, that is, $F([c_i,c_{i+1}])$ is included in the range of $f_1$ for $i=3,\cdots,u-1$. Thus, $f_i$ is well defined for each $i=3,\cdots,u-1$.
	For $u\le i\le n$,
	define $f_i:[c_i,c_{i+1}]\to [c_2,c_3]$ by
	\begin{equation}
	f_i(x):=f_2^{-1}\circ F(x),~~~~~x\in[c_i,c_{i+1}].
	\label{f2f}
	\end{equation}
Since $H(F)=1$ and $K(F)=[c_0,c_1]$, we get $F(I)\subseteq [c_0,c_1]$. In particular, $F([c_i,c_{i+1}])\subset[c_0,c_1]$ for each $u\le i\le v$.	By \eqref{f1c2} and \eqref{f2c3}, we see that the range of $f_2$ is $[c_0,c_1]$. Hence, $F([c_i,c_{i+1}])$ is included in range of $f_2$. This implies that $f_i$ in \eqref{f2f} is well defined for each $u\le i\le v$.
	
{\it Step 3}: Joint $f_i ~(0\le i\le n)$ to get a continuous iterative root of order $2$ as desired in Theorem \ref{LI}.

Let $f: I \to I$ be defined by $f(x):=f_i(x)$ on each $[c_i,c_{i+1}]$ for $i=0,\cdots,v$.
	We claim that $f$ is a continuous  iterative root of order $2$ of $F$ such that $f$ is strictly increasing on $[c_0,c_1]$ and $H(f)=2$. In fact, it is easy to see that $f_i$ is continuous on its domain for each $i=0,\cdots,v$.
	By \eqref{extencon}, \eqref{f1c1}, \eqref{f2c3}, \eqref{f1f} and \eqref{f2f}, we have
	\begin{align*}
	&f_3(c_3)=f_1^{-1}\circ F(c_3)=f_1^{-1}\circ f_0(c_1)=f_1^{-1}\circ f_1(c_1)=c_1=f_2(c_3),
	\\
	&f_i(c_{i+1})=f_1^{-1}\circ F(c_{i+1})=f_{i+1}(c_{i+1})~~ \text{for}~~3\le i\le u-2,
	\\
	&f_{u-1}(c_u) =f_1^{-1}\circ F(c_u)=f_1^{-1}(c_0)=c_2=f_2^{-1}(c_0)=f_2^{-1}\circ F(c_u)=f_u(c_u),
	\\
	&f_i(c_{i+1})=f_2^{-1}\circ F(c_{i+1})=f_{i+1}(c_{i+1})~~ \text{for}~~u\le i\le v.
	\end{align*}
Together with \eqref{f1c1} and \eqref{f1c2}, it follows that $f$ is continuous on $I$. On the other hand, $f^2(x)=F(x)$ for $x\in[c_0,c_1]$ since $f(x)=f_0(x)$.
	By \eqref{f12}, we have
	\begin{equation*}
	f^2(x) = f\circ f_0^{-1}\circ F(x)=f_0\circ f_0^{-1}\circ F(x)=F(x) ~~\text{for}~~x\in[c_1,c_3].
	\end{equation*}
	According to \eqref{f1f}, we get
	\begin{equation*}
	f^2(x) = f\circ f_1^{-1}\circ F(x)=f_1\circ f_1^{-1}\circ F(x)=F(x) ~~\text{for}~~x\in[c_3,c_u].
	\end{equation*}
	By \eqref{f2f}, we further obtain
	\begin{equation*}
	f^2(x) = f\circ f_2^{-1}\circ F(x)=f_2\circ f_2^{-1}\circ F(x)=F(x) ~~\text{for}~~x\in[c_u,c_{v+1}].
	\end{equation*}
	The above discussion implies that $f$ is an iterative root of order $2$ of $F$. Moreover, $f=f_0$ on $[c_0,c_1]$ is strictly increasing. Finally, noting  that $f([c_i,c_{i+1}])\subset [c_0,c_1]$ for $i=1,2$, $f([c_i,c_{i+1}])\subset [c_1,c_2]$ for $i=3,\cdots,u-1$ and $f([c_i,c_{i+1}])\subset [c_2,c_3]$ for $i=u,\cdots,v$, we have $f(I)\not\subset[c_0,c_1]$ but $f^2(I)\subset[c_0,c_1]$. This gives that $H(f)=2$ by Lemma \ref{height}. We complete the proof of Theorem~\ref{LI}.
\end{proof}

\begin{proof}[\bf Proof of Theorem \ref{MI}]

	{\it Step 1}: Construct a continuous and strictly increasing iterative root of $F$ of order $2$ on $[c_k,c_{k+1}]$.

Under condition \eqref{f_01}, \eqref{f_02} or \eqref{f_03}, using \cite[Theorem $15.7$]{Kuczma} and Lemma \ref{iev} again, we can construct a continuous strictly increasing function $f_k:[c_k,c_{k+1}]\to [c_k,c_{k+1}]$ such that
	\begin{equation}
	f_k(c_k)=F(c_{k+2}),~~~~~ f_k(c_{k+1})=F(c_{k-1}),
	\label{extenmid}
	\end{equation}
	and $f_k^2(x)=F(x)$ for all $x\in[c_k,c_{k+1}]$.

	{\it Step 2}: Define functions on intervals $[c_{k-1},c_k]$ and $[c_{k+1},c_{k+2}]$, respectively.

	For $i=k-1,k+1$, we define $f_i:[c_i,c_{i+1}]\to [c_k,c_{k+1}]$ by
	\begin{equation}
	f_i(x):=f_k^{-1}\circ F(x),~~~~~~x\in [c_i,c_{i+1}].
	\label{firexten}
	\end{equation}
	Note that $F(c_k)=F(c_{k+2})\ge f_k(c_k)$ and $F(c_{k+1})=F(c_{k-1})\le f_k(c_{k+1})$. Thus, by \eqref{extenmid}, $F([c_i,c_{i+1}])$ ($i=k-1,k+1$) are both included in the range of $f_k$ since $F$ is monotone on $[c_i,c_{i+1}]$ for $i=k-1,k+1$. This implies that $f_i~(i=k-1,k+1)$ in \eqref{firexten} are well defined.
	Since $f_k$ is strictly increasing and $F|_{[c_i,c_{i+1}]}$ ($i=k-1,k+1$) are  strictly decreasing, we get that $f_i$ ($i=k-1,k+1$) are both strictly decreasing. Clearly, $f_i$ ($i=k-1,k+1$) are continuous. Moreover, by \eqref{extenmid} and \eqref{firexten}, we have
	\begin{align}
	&f_{k-1}(c_{k-1})=f_k^{-1}\circ F(c_{k-1})=c_{k+1},
	\label{fk-1ck-1}
	\\
	&f_{k-1}(c_k)=f_k^{-1}\circ F(c_k)=f_k^{-1}\circ f_k^2(c_k)=f_k(c_k),
	\label{fk-1ck}
	\\
	&f_{k+1}(c_{k+1})=f_k^{-1}\circ F(c_{k+1})=f_k^{-1}\circ f_k^2(c_{k+1})=f_k(c_{k+1}),
	\label{fk+1l}
	\\
	&f_{k+1}(c_{k+2})=f_k^{-1}\circ F(c_{k+2})=c_k.
	\label{fk+1ck+2}
	\end{align}

	{\it Step 3}: Define a function on interval $[c_0,c_{k-1}]$.

	For $\ell_{j+1}\le i\le \ell_j-1$ with
	$j\equiv 0(\!\!\mod 4)$, we define
	$f_i:[c_i,c_{i+1}]\to[c_{k+1},c_{k+2}]$ as
	\begin{equation}
	f_i(x):=f_{k+1}^{-1}\circ F(x),~~~~~ x\in [c_i,c_{i+1}].
	\label{fcl1}
	\end{equation}
	By \eqref{upexten1}, \eqref{fk+1l}, \eqref{fk+1ck+2} and \eqref{extenmid}, we have $F([c_{\ell_{j+1}},c_{\ell_j}])\subset[c_k,F(c_{k-1})]=[c_k,f_k(c_{k+1})]=[c_k,f_{k+1}(c_{k+1})]$, that is, $F([c_{\ell_{j+1}},c_{\ell_j}])$ is included in the range of $f_{k+1}$, implying $f_i$ in \eqref{fcl1} is well defined.

	For $\ell_{j+1}\le i\le \ell_j-1$ with
	$j\equiv 1~\text{or}~3(\!\!\mod 4)$, we define
	$f_i:[c_i,c_{i+1}]\to[c_k,c_{k+1}]$ as
	\begin{equation}
	f_i(x):=f_k^{-1}\circ F(x),~~~~~ x\in [c_i,c_{i+1}].
	\label{ficlj13}
	\end{equation}
	By \eqref{midexten} and \eqref{extenmid}, we see that $F([c_{\ell_{j+1}},c_{\ell_j}])$ is included in the range of $f_k$. This means that each $f_i$ of \eqref{ficlj13} is well defined.

	For $\ell_{j+1}\le i\le \ell_j-1$ with
	$j\equiv 2(\!\!\mod 4)$,
	we define
	$f_i:[c_i,c_{i+1}]\to[c_{k-1},c_k]$ as
	\begin{equation}
	f_i(x):=f_{k-1}^{-1}\circ F(x),~~~~~ x\in [c_i,c_{i+1}].
	\label{ficlj2}
	\end{equation}
	One can check that each $f_i$ of \eqref{ficlj2} is well defined by \eqref{midexten}, \eqref{fk-1ck-1}, \eqref{fk-1ck} and \eqref{extenmid}.

	{\it Step 4}: Define a function on interval $[c_{k+2},c_{v+1}]$.
	
	For $r_j \le i\le r_{j+1}-1$ with
	$j\equiv 0(\!\!\mod 4)$, we define $f_i:[c_i,c_{i+1}]\to[c_{k-1},c_k]$ as in
	\eqref{ficlj2}.
 For $r_j \le i\le r_{j+1}-1$ with
 $j\equiv 1~\text{or}~3(\!\!\mod 4)$, define
 $f_i:[c_i,c_{i+1}]\to[c_k,c_{k+1}]$ as in
 \eqref{ficlj13}.
 For $r_j \le i\le r_{j+1}-1$ with
 	$j\equiv 2(\!\!\mod 4)$, define $f_i:[c_i,c_{i+1}]\to[c_{k+1},c_{k+2}]$ as in
 	\eqref{fcl1}.		
	By \eqref{midextenr1}, \eqref{midextenr2}, \eqref{fk-1ck-1}-\eqref{fk+1ck+2} and  \eqref{extenmid}, it is easy to verify that $f_i$ is well defined for each $k+2\le i\le v$.

	{\it Step 5}: Define a continuous  iterative root of order $2$ by jointing $f_i~(0\le i\le v)$.

	Let $f: I\to I$ be identical with $f_i$ on each $[c_i,c_{i+1}]$ for $0\le i\le v$.
	We claim that the function $f$ is a continuous iterative root of order $2$ of $F$ such that $f$ is strictly increasing on $[c_k,c_{k+1}]$ and $H(f)=2$. In fact, $f_i$ is continuous on its domain for every $0\le i\le v$. From \eqref{fk-1ck} and \eqref{fk+1l} we get that $f$ is continuous at $c_k$ and $c_{k+1}$. Thus, in order to prove the continuity of $f$ on $I$, it suffices to show  that $f$ is continuous at points $\{c_i\}_{i=1}^{n}\setminus\{c_k,c_{k+1}\}$. For each $\ell_{j+1}\le i\le \ell_j-2$, $0\le j\le s$, and each  $r_j\le i\le r_{j+1}-1$, $0\le j\le t$, we get
	\begin{equation*}
	f_i(c_{i+1})=f_{i+1}(c_{i+1})
	\end{equation*}
	 because $f_i$ and $f_{i+1}$ have the same expression, which is one of \eqref{fcl1}, \eqref{ficlj13} and \eqref{ficlj2}. It follows that $f$ is continuous at points
	 \begin{equation*}
\cup_{0\le j\le s} \{c_i\}_{i=\ell_{j+1}+1}^{\ell_j-1}~~~\text{and}~~~\cup_{0\le j\le t} \{c_i\}_{i=r_j+1}^{r_{j+1}-1}.
	 \end{equation*}
In order to show the continuity of $f$, it remains to prove that $f$ is continuous at points $\{c_{\ell_j}\}_{j=0}^s$	 and $\{c_{r_j}\}_{j=0}^t$.
	Note that $f_{k+1}(c_{k+1})=f_k(c_{k+1})=F(c_{k-1})$ by \eqref{fk+1l} and \eqref{extenmid}, implying that
	\begin{equation}
	f_{k+1}^{-1}\circ F(c_{k-1})=f_k^{-1}\circ F(c_{k-1})=c_{k+1},
	\label{ck+1}
	\end{equation}
	and $f_{k-1}(c_k)=f_k(c_k)=F(c_{k+2})$ by \eqref{fk-1ck} and  \eqref{extenmid}, implying that
	\begin{equation}
	f_{k-1}^{-1}\circ F(c_{k+2})=f_k^{-1}\circ F(c_{k+2})=c_k.
	\label{ck}
	\end{equation}
Thus, for each $0\le j\le s$ with $j\equiv0(\mod4)$, we get
\begin{align}
f_{\ell_j-1}(c_{\ell_j})&=f_{k+1}^{-1}\circ F(c_{\ell_j})=f_{k+1}^{-1}\circ F(c_{k-1})=c_{k+1},
\label{conti0l}
\\
f_{\ell_j}(c_{\ell_j})&=\begin{cases}
f_{k-1}(c_{k-1})=c_{k+1}~~~& \text{for}~ j=0,\\
f_k^{-1}\circ F(c_{\ell_j}) = f_k^{-1}\circ F(c_{k-1})=c_{k+1}& \text{for}~ j>0.
\end{cases}
\label{conti0r}
\end{align}
Actually, \eqref{conti0l} is obtained by \eqref{fcl1}, \eqref{jointp} and \eqref{ck+1}, and \eqref{conti0r} is obtained by \eqref{fk-1ck-1}, \eqref{ficlj13}, \eqref{jointp} and \eqref{ck}.
This implies that $f$ is continuous at $c_{\ell_j}$, where $j\equiv0(\!\!\mod4)$.
For $0\le j\le s$ with $j\equiv1(\!\!\mod4)$, we have
\begin{align*}
f_{\ell_j-1}(c_{\ell_j})=f_k^{-1}\circ F(c_{\ell_j}) = f_k^{-1}\circ F(c_{k-1})=c_{k+1}~~&\text{by}~~\eqref{ficlj13}, \eqref{jointp}, \eqref{ck+1},
\\ f_{\ell_j}(c_{\ell_j})=f_{k+1}^{-1}\circ F(c_{\ell_j})= f_{k+1}^{-1}\circ F(c_{k-1})=c_{k+1}~~&\text{by}~~\eqref{fcl1}, \eqref{jointp}, \eqref{ck+1}.
\end{align*}
This shows that $f$ is continuous at $c_{\ell_j}$, where $j\equiv1(\!\!\mod4)$.
For $0\le j\le s$ with $j\equiv2(\!\!\mod4)$, we have
\begin{align*}
f_{\ell_j-1}(c_{\ell_j})=f_{k-1}^{-1}\circ F(c_{\ell_j})=f_{k-1}^{-1}\circ F(c_{k+2})=c_k~~&\text{by}~~\eqref{ficlj2}, \eqref{jointp}, \eqref{ck},
\\
f_{\ell_j}(c_{\ell_j})=f_k^{-1}\circ F(c_{\ell_j})= f_k^{-1}\circ F(c_{k+2})=c_k~~&\text{by}~~\eqref{ficlj13}, \eqref{jointp}, \eqref{ck}.
\end{align*}
This shows that $f$ is continuous at $c_{\ell_j}$, where $j\equiv2(\!\!\mod4)$.
For $0\le j\le s$ with $j\equiv3(\!\!\mod4)$, we get
\begin{align*}
f_{\ell_j-1}(c_{\ell_j})=f_k^{-1}\circ F(c_{\ell_j})= f_k^{-1}\circ F(c_{k+2})=c_k~~&\text{by}~~\eqref{ficlj13}, \eqref{jointp}, \eqref{ck},
\\
f_{\ell_j}(c_{\ell_j})=f_{k-1}^{-1}\circ F(c_{\ell_j})= f_{k-1}^{-1}\circ F(c_{k+2})=c_k~~&\text{by}~~ \eqref{ficlj2}, \eqref{jointp}, \eqref{ck}.
\end{align*}
This shows that $f$ is continuous at $c_{\ell_j}$, where $j\equiv3(\!\!\mod4)$.
Similarly, by \eqref{extenck+2cr1}, \eqref{fcl1}-\eqref{ficlj2}, \eqref{ck+1} and \eqref{ck}, we can prove that $f$ is continuous at $c_{r_j}$ for all $0\le j\le t$.
Thus, $f$ is continuous on $I$.

Now, we show that $f$ is an iterative roots of order $2$ of $F$. For every $x\in [c_k,c_{k+1}]$, we have $f^2(x)=F(x)$ because $f=f_k$ and $f_k$ is an  iterative root of order $2$ of $F|_{[c_k,c_{k+1}]}$. For $x\in[c_i,c_{i+1}]$ and $i=k-1,k+1,\ell_{j-1},\cdots,\ell_j-1~\text{with}~j\equiv1~\text{or}~3\\(\!\!\mod 4)$ and  $r_j,\cdots,r_{j+1}-1~\text{with}~j\equiv1~\text{or}~3(\!\!\mod 4)$, we have
	\begin{equation*}
	f^2(x)= f_k\circ f_k^{-1}\circ F(x)=F(x).
	\end{equation*}
	For $x\in[c_i,c_{i+1}]$, $i=\ell_{j-1},\cdots,\ell_j-1$ with $j\equiv2(\!\!\mod4)$ and $ r_j,\cdots,r_{j+1}-1$ with $j\equiv0(\!\!\mod4)$, we have
	\begin{equation*}
	f^2(x)=f_{k-1} \circ f_{k-1}^{-1}\circ F(x)=F(x).
	\end{equation*}
	For $x\in[c_i,c_{i+1}]$, $i=\ell_{j-1},\cdots,\ell_j-1$ with $j\equiv0(\!\!\mod4), r_j,\cdots,r_{j+1}-1$ with $j\equiv2(\!\!\mod4)$, we have
	\begin{equation*}
	f^2(x)=f_{k+1}\circ f_{k+1}^{-1} \circ F(x)=F(x).
	\end{equation*}
Thus, it is proved that $f$ is an iterative root of order $2$ of $F$. Clearly, $f=f_k$ on $[c_k,c_{k+1}]$ is strictly increasing. Finally, by \eqref{firexten} and
\eqref{ficlj13}, $f([c_i,c_{i+1}]) \subset [c_k,c_{k+1}]$ for $i=k-1,k+1,\ell_{j-1},\cdots,\ell_j-1~\text{with}~j\equiv1~\text{or}~3(\!\!\mod 4), r_j,\cdots,r_{j+1}-1~\text{with}~j\equiv1~\text{or}~3(\!\!\mod 4)$, $f([c_i,c_{i+1}]) \subset [c_{k-1},c_k]$ for $i=\ell_{j-1},\cdots,\ell_j-1$ with $j\equiv2(\!\!\mod4)$, $ r_j,\cdots,r_{j+1}-1$ with $j\equiv0(\!\!\mod4)$ by \eqref{ficlj2} and $f([c_i,c_{i+1}]) \subset [c_{k+1},c_{k+2}]$ for $i=\ell_{j-1},\cdots,\ell_j-1$ with $j\equiv0(\!\!\mod4), r_j,\cdots,r_{j+1}-1$ with $j\equiv2(\!\!\mod4)$ by \eqref{fcl1}. Thus, $H(f)=2$ by Lemma \ref{height}. This proves the claim and we complete the whole proof.
\end{proof}

\begin{proof}[\bf Proof of Theorem \ref{LD}]
{\it Step 1:} Construct a continuous and strictly decreasing iterative root $f_0$ of order $2$ of $F$ on $[c_0,c_1]$.

Since $F$ is reversing-correspondence on $[c_0,c_1]$, by \eqref{ldc}, Theorem $15.9$ of \cite{Kuczma} and Lemma \ref{diteendps}, we can construct a continuous and strictly decreasing function $f_0:[c_0,c_1]\to[c_0,c_1]$ such that
\begin{equation}
f_0(c_0)=F(c_1),~~~~~f_0(c_1)=c_0,
\label{f0endps}
\end{equation}
and $f_0^2(x)=F(x)$ for all $x\in[c_0,c_1]$.

{\it Step 2}: Define a function on $[c_1,c_2]$.

Define $f_1:[c_1,c_2]\to[c_0,c_1]$ by
\begin{equation}
f_1(x):=f_0^{-1}\circ F(x),~~~~~ x\in [c_1,c_2].
\label{ldf1}
\end{equation}
Since $f_0$ is strictly decreasing on $[c_0,c_1]$ and $F$ is strictly decreasing on $[c_1,c_2]$, by \eqref{ldc} and \eqref{f0endps} we have $F([c_1,c_2])=[c_0,F(c_1)] = [f_0(c_1),f_0(c_0)]$, that is, $F([c_1,c_2])$ is included in the range of $f_0$, implying that $f_1$ in \eqref{ldf1} is well defined.

{\it Step 3}: Define a function on each $[c_i,c_{i+1}]$ for $2\le i\le v$.

For each $2\le i\le v$, define a function $f_i:[c_i,c_{i+1}]\to[c_1,c_2]$ by
\begin{equation}
f_i(x):=f_1^{-1}\circ F(x),~~~~~x\in[c_i,c_{i+1}].
\label{ldfi}
\end{equation}
Note that $f_1$ is strictly increasing, $f_1(c_1)=c_0$ and $f_1(c_2)=c_1$. This implies that $f_i$ in \eqref{ldfi} is well defined since $F([c_i,c_{i+1}])\subset[c_0,c_1]$.

{\it Step 4}: Joint $f_i~(0\le i\le v)$ to give an iterative root of order $2$ as desired in Theorem \ref{LD}.

Let $f(x):=f_i(x)$ for all $x\in[x_i,x_{i+1}]$ and $0\le i\le v$. Now we check that $f$ is a continuous  iterative root of order $2$ of $F$ such that $f$ is strictly decreasing on $[c_0,c_1]$ and $H(f)=2$. Clearly, $f_i$ is continuous on its domain for all $0\le i\le v$. In order to prove the continuity of $f$ on $I$, it remains to show that $f$ is continuous at points $\{c_i\}_{i=1}^v$. Hence, by \eqref{f0endps} and \eqref{ldf1}, we get $f_1(c_1)=f_0^{-1}\circ F(c_1)=c_0=f_0(c_1)$. This implies that $f$ is continuous at $c_1$. Furthermore, by \eqref{ldf1}, \eqref{ldc} and \eqref{ldfi}, we have
$$
f_1(c_2)=f_0^{-1}\circ F(c_2)=f_0^{-1}(c_0)=c_1=f_1^{-1}(c_0)=f_1^{-1}\circ F(c_2)=f_2(c_2).
$$
This shows that $f$ is continuous at $c_2$. Note that  \eqref{ldfi} yields $f_i(c_{i+1})=f_1^{-1}\circ F(c_{i+1})=f_{i+1}(c_{i+1})$ for every $2\le i\le v-1$. It follows that $f$ is continuous at $\{c_i\}_{i=3}^v$. Hence, we proved the continuity of  $f$ on $I$.

On the other hand, it is easy to see that $f^2(x)=F(x)$ for all $x\in[c_0,c_1]$ since $f=f_0$ and $f_0$ is an iterative root of order $2$ of $F$ on $[c_0,c_1]$, as shown in {\it Step 1}. For every $x\in [c_1,c_2]$, by \eqref{ldf1} we have
$$
f^2(x)=f_0\circ f_0^{-1}\circ F(x) =F(x).
$$
For $x\in [x_2,c_{v+1}]$, we have
$$
f^2(x)=f_1\circ f_1^{-1}\circ F(x) =F(x)
$$
by \eqref{ldfi}. This proves that $f$ is an iterative root of order $2$ of $F$.

Note that $f$ is strictly decreasing on $[c_0,c_1]$ since $f=f_0$ on $[c_0,c_1]$.
By \eqref{ldf1} and  \eqref{ldfi}, we have $f([c_1,c_2])\subset[c_0,c_1]$ and $f([c_i,c_{i+1}])\subset[c_1,c_2]$ for all $2\le i\le v$, yielding that $f^2(I)\subset[c_0,c_1]$ but $f(I)\not\subset[c_0,c_1]$. It follows that $H(f)=2$ by Lemma \ref{height}. The proof is completed.
\end{proof}

\begin{proof}[\bf Proof of Theorem 2.4]
{\it Step1}: Construct a continuous and strictly decreasing iterative root $f_k$ of order $2$ of $F$ on $[c_k,c_{k+1}]$.

Note that $F$ is reversing-correspondence on $[c_k,c_{k+1}]$. Under \eqref{mdc}, \eqref{mdc2} or \eqref{mdc3}, we infer from Theorem $15.9$ of \cite{Kuczma} and Lemma \ref{diteendps} that there exists a continuous and strictly decreasing function $f_k:[c_k,c_{k+1}]\to[c_k,c_{k+1}]$ such that
\begin{equation}
f_k(c_k)=F(c_{k-1}),~~~~~f_k(c_{k+1})=F(c_{k+2}),
\end{equation}
and $f_k^2(x)=F(x)$ for all $x\in[c_k,c_{k+1}]$.

{\it Step 2}: Define functions on intervals $[c_{k-1},c_k]$ and $[c_{k+1},c_{k+2}]$, respectively.

{\it Step 3}: Define a function on interval $[c_0,c_{k-1}]$.

{\it Step 4}: Define a function on interval $[c_{k+2},c_{v+1}]$.

{\it Step 5}: Define a continuous  iterative root of order $2$ by jointing above functions. Since {\it Step 2} -{\it Step 5} are similar to those in the proof of Theorem \ref{MI}, we omit their details. The proof of Theorem \ref{MD} is completed.
\end{proof}

\end{document}